\documentclass{amsart}%
\usepackage{amsfonts}
\usepackage{amsmath, amsthm}
\usepackage{amssymb}
\usepackage{graphicx}%
\newtheorem{theorem}{Theorem}
\theoremstyle{plain}

\newtheorem{definition}{Definition}

\newtheorem{proposition}{Proposition}
\newtheorem{remark}{Remark}

\numberwithin{equation}{section}
\begin{document}
\title[ ]{Hardy-Littlewood-Polya's inequality and a new concept of weak majorization}
\author{Ionel Roven\c{t}a}
\address{Department of Mathematics, University of Craiova, Craiova 200585, Romania}
\email{ionelroventa@yahoo.com}
\thanks{}

\begin{abstract}
In this paper we study some weak majorization properties with  applications for the trees. A strongly notion of majorization is
introduced and Hardy-Littlewood-Polya's inequality is generalized. 
\end{abstract}
\maketitle

\section{Introduction}

In last years, a lot of papers was dedicated to majorization theory,
that was scattered in journals in a wide variety of fields. Indeed, many majorization concepts have been
reinvented and used in different research areas, as Lorenz or dominance ordering in economics,
optimization and graph theory.

Whenever the solution of a problem involves a discrete uniform distribution,
the idea of a majorization proof was intensively used.
Moreover, if a uniform allocation or distribution was in a sense optimal,
then the concept of majorization frequently can be used to order  allocations or distributions.

Naturally extensions of the majorization concept are possible and indeed many of them have been fruitfully introduced.
The aim of this paper is to introduce a new majorization concept, from which derives  multiple applications
in different areas.

Firstly, we remind the classical majorization concept. For $x\in \mathbb{R}^N$, let denote by $x_{[i]}$ the $i$th largest component of  the vector $x$.

\begin{definition}
If $x,y\in \mathbb{R}^N$
we say that $x$ is weakly majorized by $y$, denoted $x \prec_{*} y$, provided that
\[
\sum_{i=1}^{k}x_{[i]}\leq \sum_{i=1}^{k} y_{[i]}, \; k=1,2,...,N.
\]

Moreover, if $\sum_{i=1}^{N}x_{[i]}=\sum_{i=1}^{N}y_{[i]}$ holds, then we say that  $x$ \emph{is majorized by} $y$ and is denoted by $x\prec y$.
\end{definition}

Further details and applications about majorization can be found in  Marshall and Olkin \cite{MO}.
In fact, majorization turns out to be an essential structure for several classes of symmetric inequalities.
For example, the arithmetic-geometric mean inequality, Jensen's inequality, Hermite-Hadamard's inequality and Hardy-Littlewood-Polya's inequality
can be easily derived by using an argument based on   Schur-convex functions theory.
Moreover, there are many connections
with matrix theory, more exactly with doubly stochastics matrices, i.e.  nonnegative matrices with all rows and columns sums are equal to one.
Recall here  a result  from \cite{MO}.

\begin{theorem}\label{equiv}
Let $x,y\in \mathbb{R}^N$, then the following statements are equivalent:

i) $x\prec y$;

ii) There is a doubly stochastic matrix $A$ such that $x=Ay$;

iii) The inequality $\sum_{i=1}^{N}f(x_{i})\leq \sum_{i=1}^{N}f(y_{i})$, holds for each convex function $f:\mathbb{R}\rightarrow \mathbb{R}$.
\end{theorem}

The statement $iii)$ is precisely  Hardy-Littlewood-Polya's inequality. A slide generalization of iii)
is presented in  Tomi\'{c} \cite{MT} and Weyl \cite{HW}.

\begin{theorem}\label{TW}
If $x\prec_{*} y$ then for every nondecreasing convex function $f:\mathbb{R}\rightarrow \mathbb{R}$ we have
\[
\sum_{i=1}^{N}f(x_{i})\leq \sum_{i=1}^{N}f(y_{i}).
\]
\end{theorem}

Based on the above results it can be easy seen that there is a connection between majorization and doubly stochastic matrices,
which can be strongly related to graph theory and some network  problems. 
The concept of majorization is strongly related with a special type of convex function. 
The Schur-convex functions was introduced by I. Schur in 1923 and
have important applications in analytic inequalities, elementary
quantum mechanics and quantum information theory.

\begin{definition}
The function $F: A\rightarrow\mathbb{R}$, where
$A\subset\mathbb{R}^{N}$, is called Schur-convex if for each $x,y\in A$ with $x\prec y$ we have $F(x)\leq F(y)$.
Any such function $F$ is Schur-concave if  $-F$ is a  Schur-convex function.
\end{definition}

An important source of Schur-convex functions was given in
\cite{MM} by Merkle. Guan in \cite{KG} -\cite{KG1}
proved that all symmetric elementary functions and the symmetric means of order $k$
are Schur-concave functions. Other interesting examples of Schur-convex functions can
be found in \cite{CZW1, CZW3} and \cite{IR2010}.  Li and  Trudinger \cite{LT} consider a special class of
inequalities for elementary symmetric functions that are relevant in the study of partial differential equations associated with curvature
problems. In this context, in \cite{NRov2014} we define the concept of majorization and we prove Hardy-Littlewood-Polya's inequality into the spaces with nonpositive curvature.

Recall here the main result,  concerning the study of  Schur-convexity of a function,  in the case of smoothness properties of the function under attention. See  \cite{NP2006} and \cite{RV} .

\begin{theorem}
\label{JMAAthm1}Let $F(x)=F(x_{1},...,x_{N})$ be a
symmetric function with continuous partial derivatives on
$I^{N}=I\times I\times...\times I$, where $I$ is an open interval.
Then $f:I^{N}\rightarrow\mathbb{R}$ is Schur-convex if and only if the following inequality
\begin{equation}
(x_{i}-x_{j})\left(  \frac{\partial f}{\partial
x_{i}}-\frac{\partial
f}{\partial x_{j}}\right)  \geq0  \label{eqsc1},
\end{equation}
holds on $I^{N}$, for each $i,j\in\{1,..,N\}$. It is strictly convex if inequality (\ref{eqsc1}) is
strict for $x_{i}\neq x_{j},\,1\leq i,j\leq N$.

Any such function $F$ is Schur-concave if the inequality (\ref{eqsc1}) is reversed.
\end{theorem}

In \cite{XMZAMS} was established a class of analytic inequalities for some special Schur-convex functions which is
related with the solutions of a second order nonlinear differential equation.
These analytic inequalities are used to infer some geometric inequalities, such as isoperimetric inequality. See also \cite{IR2012}.

On the other hand,  Schur-convexity is very useful in optimization problems.  For more details, see the  family of symmetric functions
with applications in optimization and fully nonlinear elliptic equations, such as Monge-Amp\`{e}re equation ( Theorem 6.4, Corollary 6.5 in \cite{NKR}).

The outline of the paper is as follows.  After some preliminaries on majorization and Schur-convex functions from Section 1, 
in   Section 2 we  introduce a new concept of weak majorization and we prove an interesting extension of   Hardy-Littlewood-Polya's inequality. 
In section 3  we  present some  applications into the context of some special class of trees,  the $K$-spiders.

\section{A new weak majorization concept}

In this section we introduce a new majorization concept, which represents a stronger version of weakly majorization.

\begin{definition}
If $x,y\in \mathbb{R}^N$ we say that $x$ is strongly majorized by $y$, denoted by $x \ll y$, if for each $ k=1, 2, \dots, N-1$ 
we have that
\begin{eqnarray*}\label{strong1}
 \sum_{i=1}^k x_{[i] }\sum_{i=k+1}^N y_{[i]}\leq \sum_{i=1}^{k} y_{[i]} \sum_{i=k+1}^N x_{[i]},
\end{eqnarray*}
\begin{eqnarray*} 
\sum_{i=1}^Nx_{[i]} \leq \sum_{i=1}^N y_{[i]}.
\end{eqnarray*}
\end{definition}

In order to simplify the calculus, in the rest of the paper we will use $x_i$ instead of $x_{[i]}$.

\

\begin{remark}
1) Note that, if none of the above sums are null,  we can rewrite  the above inequalities in the following compact form
\begin{equation}\label{strong}
{\bf \frac{x_{1}}{y_{1}}\leq \frac{x_{2}+...+x_{N}}{y_{2}+\dots+y_{N}},\dots,
\frac{x_{1}+...+x_{N-1}}{y_{1}+...+y_{N-1}}\leq \frac{x_{N}}{y_{N}},\; \frac{x_{1}+...+x_{N}}{y_{1}+...+y_{N}}\leq 1.}
\end{equation}
\end{remark}

\begin{proposition}
 If $x,y\in \mathbb{R}^N_{+}$ then  $x\ll y$ implies   $x\prec_{*} y$.
\end{proposition}
\begin{proof}
O moment of reflection gives us  the fact that $x\ll y$ implies the following sequence of inequalities
\[
x_1\leq \alpha \, y_1,\; x_1+x_2\leq \alpha\, (y_1+y_2),..., x_1+...+x_N=\alpha\, (y_1+...+y_N),
\]
where $\alpha=\frac{x_1+...+x_N}{y_1+...+y_N}$.

Since $\alpha \in [0,1]$ we deduce immediately that $x\prec_* y$.

\end{proof}

\begin{remark}
1) If the last inequality  from (\ref{strong}) is an equality, i. e. $\alpha=1$,
then $x\ll y$ is equivalent with $x\prec y$. 

2) On the other hand,  $x\prec_{*} y$ does not implies $x\ll y$. For instance, $(2,0,-1)\prec_* (3,2,1)$,  which not implies that
$(2,0,-1)\ll (3,2,1)$.

3) Even in the case  of vectors with positive coordinates, the two version of weak majorization are not equivalent (since $\alpha$ depends, implicitly,  on the coordinates).
\end{remark}

\

We are now in position to present the main result of the paper, which  consists of a  nice generalization of
Hardy-Littlewood-Polya's inequality, for the case of weak majorization.

\begin{theorem}\label{HLPs}
Let $x,y\in \mathbb{R}^N_{+}$ be two vectors such that  $x \ll y$ and let $\alpha=\frac{x_1+...+x_N}{y_1+...+y_N}$. Then for each convex function $f:\mathbb{R}_+\rightarrow \mathbb{R}$ we have that
\begin{equation}\label{HLPi}
\frac{f(x_1)+...+f(x_N)}{N}\leq  \alpha \frac{f(y_1)+...+f(y_N)}{N} + (1-\alpha)f(0).
\end{equation}
\end{theorem}

\begin{proof}
From (\ref{strong}) we have the same sequence of inequalities
\[
x_1\leq \alpha y_1,\; x_1+x_2\leq \alpha(y_1+y_2),\;....,\; x_1+...+x_N=\alpha(y_1+...+y_N).
\]

Hence, we have that $x\prec \alpha\, y$ and applying classical Hardy-Littlewood-Polya's inequality for the convex function $f$, we have that
\begin{equation}\label{HLPa}
\sum_{k=1}^{N}f(x_k)\leq \sum_{k=1}^{N}f(\alpha y_k).
\end{equation}

Since $f$ is a convex function, by taking into account that $\alpha\in [0,1]$ we have that
\begin{equation}\label{convk}
f(\alpha y_k)=f(\alpha y_k + (1-\alpha)0)\leq \alpha f(y_k)+ (1-\alpha)f(0) \qquad (k=1,\dots,N).
\end{equation}

By using \eqref{HLPa} and \eqref{convk} we obtain that
\[
f(x_1)+...+f(x_N)\leq \alpha\, \left(f(y_1)+...+f(y_N)\right)+ N(1-\alpha)f(0),
\]
hence it follows that \eqref{HLPi} holds.
\end{proof}

\begin{remark}
1) If we consider in Theorem \ref{HLPs} the condition $x\prec y$,  i. e.  $\alpha=1$, then \eqref{HLPi} is exactly  
Hardy-Littlewood-Polya's inequality.

2) We have a strongly Hardy-Littlewood-Polya's type inequality for  convex functions with  $f(0)=0$, of the form
\[
f(x_1)+...+f(x_N)\leq \alpha \, \left(f(y_1)+É+f(y_N)\right). 
\]
Hence, in our case (when $f(0)=0$), we have just proved that the result of Tomi\'{c}  and Weyl (Theorem \ref{TW})  holds even if we do not
suppose the fact that $f$ is a nondecreasing function.

3)  Moreover, if in  addition we consider that  $f$ is a nondecreasing, using the fact that $f(0)\leq f(\alpha y)\leq f(y)$
in the right hand term of (\ref{HLPi}), we obtain the classical result of Tomic and Weyl from Theorem \ref{TW}.

\end{remark}

\section{Applications at the trees}

This section is devoted to some applications  of the new majorization concept in graph theory, more exactly for the trees. 

Firstly, we present a brief introduction in graph theory and we
discuss how is used majorization to obtain valuable  applications on this topic.  See \cite{BLW} and \cite{KBR}, for more details about graph theory and  special  vertices in a tree.
In this context, it is interesting to consider majorization in trees as is introduced and discussed in \cite{GD}. 

Let $T=(V, E)$ be a tree, where $V$ is the set of vertices of the tree and $E$ is the set of edges of the tree. 
Let $d(u,v)$ be the distance between two vertices $u$ and $v$ in a tree $T$, which is defined as the number of edges in the
unique path from $u$ to $v$. For $u\in V$ we define the distance vector of $u$ as $d(u, \cdot)=(d(u,v): v\in V)\in \mathbb{Z}_{+}^{N}$.

\begin{definition}
Let $u,v$ two vertices in a tree $T$. We say that $u$ is \emph{weakly majorized} by $v$, denoted by  $u\prec_{*} v$, if $d(u,\cdot)\prec_{*} d(v,\cdot)$.
We say that $u$ and $v$ is majorization equivalent if $u\prec_{*} v$ and $v\prec_{*} u$. 
\end{definition}

In general, there exist at most two adjacent majorization equivalent vertices,
and in this case we say that $T$ is $m-symmetric$. See \cite{GD}.
Consider a tree $T$ and two adjacent vertices $u$ and $v$ in $T$. If we remove the edge $[u,v]$ we obtain two subtrees $T(u;v)$ and $T(v;u)$, where $u\in T(u;v)$
and $v\in T(v;u)$. Also denote by $V(u,v)$ the set of vertices of $T(u,v)$.  

In \cite{GD} it was proved that if $u\prec_{*}v$ then we have $u\prec_{*} v \prec_{*} v'$, for all $v'\in V(v,u)$. Moreover, it was introduced the following
concept, as a majorization center of a tree.

\begin{definition}
The majorization center is a vertex set, denoted by $M_T$ and is defined as follows:

i) If $T$ is $m-symmetric$ then $M_T=\{u,v\}$, where $u$ and $v$ are the two adjacent majorization-equivalent vertices.

ii) If $T$ is not $m-symmetric$ then $M_{T}$ is the intersection of all vertex sets $V(u;v)$ taken over all adjacent vertices $u$ and $v$ for which the majorization
$u\prec_{*} v$ holds.
\end{definition}

Now, using the our strongly  majorization concept, we introduce the corresponding \emph{strongly majorization center} in a tree.

\begin{definition}
Let $u,v$ two vertices in a tree $T$. We say that $u$ is strongly majorized by $v$, denoted by  $u\ll v$, if $d(u,\cdot)\ll d(v,\cdot)$.
We say that $u$ and $v$ is strongly majorization equivalent if $u\ll v$ and $v\ll u$.
\end{definition}

In a similar way as in we can define the \emph{strongly majorization center}  $M_T^s$ which allow us to present a strongly  version of Proposition 2 from \cite{GD}. 

\begin{proposition}\label{maj1}
Let $T=(V,E)$ be a tree with $|V|=N$ and $f:\mathbb{R}^N_{+}\rightarrow \mathbb{R}$ be a Schur-convex function, which is nondecreasing function in each variable. Define $F:V\rightarrow \mathbb{R}$
by $F(v)=f(d(v,\cdot))$ for each $v\in V$. Then there exists a vertex $v_0\in M_T^s$ such that
\begin{equation}\label{min1}
F(v_0)\leq \alpha F(v),\; \forall v\in V,
\end{equation}
where $\alpha=\frac{\sum_{i=1}^{N}d(v_0,x_i)}{\sum_{i=1}^{N}d(v,x_i)}<1$, and $x_i\in V$ are the vertices of the tree.

In particular, this conclusion holds if $f$ is of the form $f(x)=\sum_{i=1}^{N} g(x_i)$, where $g:\mathbb{R}\rightarrow \mathbb{R}$ is convex and nondecreasing.
\end{proposition}

\begin{proof} In order to obtain (\ref{min1}) the proof is similar with the proof  of Proposition 2 from \cite{GD}, by using in addition  Theorem \ref{HLPs}.
\end{proof}

The above  result
is an application of majorization at the equity in location analysis.  For instance, when locating a public facility, the majorization concept offers a good strategy to optimize the distribution
of the distances to its customers.

Erkut \cite{EE} studied several different inequalities measure, or equity, for the distribution $d=(d_i(x):i\leq N)$ of the distances
between a facility $x$ and its customers $i=1,2,...,N$. In this sense, an example is the measure $\sum_{i=1}^{N}(d_i- d_m)^2$, where $d_m$ is the
average $d=\frac{1}{N}\sum_{i=1}^{N} d_i$. One may then look for a facility location which minimizes the selected inequality measure, in some sense,
that the distances the different customers are as equals as possible.

\subsection{The convexity on the trees}

This subsection is devoted to the study of convex functions defined on special type of trees, the $K$-spiders, which are  spaces with curved geometry.
A formal definition of  the  spaces with global nonpositive curvature (abbreviated, global NPC spaces)
is as follows:

\begin{definition}
A global NPC space is a complete metric space $E=(E,d)$ for which the
following inequality holds true: for each pair of points $x_{0},x_{1}\in E$
there exists a point $y\in E$ such that for all points $z\in E,$%
\begin{equation}
d^{2}(z,y)\leq\frac{1}{2}d^{2}(z,x_{0})+\frac{1}{2}d^{2}(z,x_{1})-\frac{1}%
{4}d^{2}(x_{0},x_{1}). \label{NPC}%
\end{equation}

\end{definition}

These spaces are also known as the Cat 0 spaces. See \cite{Sturm}. In a global
NPC space, each pair of points $x_{0},x_{1}\in E$\ can be connected by a
geodesic (that is, by a rectifiable curve $\gamma:[0,1]\rightarrow E$ such
that the length of $\gamma|_{[s,t]}$ is $d(\gamma(s),\gamma(t))$ for all
$0\leq s\leq t\leq1)$. Moreover, this geodesic is unique. The point $y$ that
appears in Definition 1 is the \emph{midpoint} of $x_{0}$ and $x_{1}$ and has
the property
\[
d(x_{0},y)=d(y,x_{1})=\frac{1}{2}d(x_{0},x_{1}).
\]

Every Hilbert space is a global NPC space. Its geodesics are the line segments.
A Riemannian manifold $(M,g)$ is a global NPC space if and only if it is
complete, simply connected and of nonpositive sectional curvature.

Recently, in \cite{NRov2014} was defined the majorization concept  into such spaces with global nonpositive curvature. It was proved that, if we have a characterization 
of convex functions in such spaces, then combining with a symmetry of coordinates we deduce immediately  the Schur-convexity property of the function under attention. 
For more details, see \cite{NRov2014}.

Let  $K$ be an arbitrary set and for each $i\in K$ we denote by  $N_{i}=\{(i,r):r\in \mathbb{R}\}$ a duplicate of $\mathbb{R_{+}}$
with the usual metric. We define the  $K$-spider $(N,d)$, as the reunion of the spaces  $N_{i},\,i\in K$, in their origins, 
which means
\begin{equation*}
N=\{(i,r):i\in K,\, r\in \mathbb{R} \}/\sim, \;\; where\; (i,0)\sim (j,0)\;
(\forall i, j \in K)
\end{equation*}
and
\begin{equation}\label{met}
d((i,r),(j,s))= \left\{%
\begin{array}{lll}
|r-s|,  \;\mbox{if} \; i=j; \\
|r|+|s|,\;  \mbox{otherwise}.
\end{array}%
\right.
\end{equation}
The sets $N_i$ can be seen as  closed subsets of $N$.
The intersection of each two sets  $N_i$ and $N_j$, with $i\neq j$, is given by the origin  $o:= (i,0)=(j,0)$.

The $K$-spider $N$ depends only on the cardinality of  $ K$. If the set  $%
K=\{1,2,...,k\}$  the $K$-spider is renamed as a $k-spider$. The tripod  is a $3-spider$.

\begin{proposition}
 Each $K$-spider endowed with the metric given by \eqref{met} is a global NPC space.
\end{proposition}
\begin{proof}
See \cite{Sturm}.
\end{proof}

We are now in position to study the convex functions defined on a tripod $N$,  with the edges
 $N_{1}$, $N_{2}$, $N_{3}.$  We consider a function  $f:N\rightarrow \mathbb{R}$  and let $f_{1}:[0,\infty) \rightarrow \mathbb{R}$, $f_{2}:[0,\infty) \rightarrow \mathbb{R}$ 
 and $f_{3}:[0,\infty) \rightarrow \mathbb{R}$ be the  restrictions of    $f$ to each edge  $N_{1}$, $N_{2}$, $N_{3}.$ (i. e. $f((i,k))=f_i(k),  \forall  k\in \mathbb{R}$, $i=1,\dots, 3.$)
 
 The main problem consist in finding the conditions which should be imposed to the functions  $f_1,$ $f_2$ and $f_3$  such that  the function $f$ 
 is convex. Of course, the restrictions $f_1,$ $f_2$, $f_3$ need to be convex. The following result gives an answer of  this problem.

\begin{theorem}
Let $f:N\rightarrow \mathbb{R}$ be a function  and let $f_{1}:[0,\infty) \rightarrow \mathbb{R}$, $f_{2}:[0,\infty) \rightarrow \mathbb{R}$ 
 and $f_{3}:[0,\infty) \rightarrow \mathbb{R}$  be the corresponding restrictions which verify 
 $$
 f((i,k))=f_i(k)  \qquad (  k\in \mathbb{R},\,  i=1,\dots, 3),
 $$
 $$
 f_1(0)=f_2(0)=f_3(0).
 $$
 If  the functions $f_1$, $f_2$ and $f_3$ are convex and the functions  $f_1+f_2$,
 $f_1+f_3$ and $f_2+f_3$ are nondecreasing 
 then $f$ is a convex function.
\end{theorem}

\begin{proof}
For simplicity we give the proof when $f$ is a differentiable function. 

Note that the conditions which says that $f_1$, $f_2$ and $f_3$ are convex functions and  the functions $f_1+f_2$,
 $f_1+f_3$ and $f_2+f_3$ are nondecreasing,  mean exactly that the sum of the right derivatives are positives, i. e.
\begin{equation}
f_{1d}^{^{\prime }}(0)+f_{3d}^{^{\prime }}(0)\geq 0,  \label{cneq3}
\end{equation}%
\begin{equation}
f_{1d}^{^{\prime }}(0)+f_{2d}^{^{\prime }}(0)\geq 0,  \label{cneq4}
\end{equation}%
\begin{equation}
f_{2d}^{^{\prime }}(0)+f_{3d}^{^{\prime }}(0)\geq 0.  \label{cneq5}
\end{equation}

Whithout loosing the generality we can suppose that  $f_1(0)=f_2(0)=f_3(0)=0$.
 Let $x_1=(1,a)\in N_1,\, a>0$, $x_2=(2,b)\in N_2,\, b>0$, $x_3=(3,c)\in N_3,\, c>0$
be three points belonging to different edges of the tripod.  We consider the measure $\mu= \lambda_1
\delta_{x_1}+ \lambda_2 \delta_{x_2}+ \lambda_3 \delta_{x_3}$,  where $%
\lambda_1+ \lambda_2+ \lambda_3=1,\, \lambda_i>0, \, i=1,\dots, 3$ and $\delta_x $ is the Dirac measure concentrated at the point $x$. 

In order to prove the convexity of the function $f$ we need to prove that  the following inequality hold
\begin{equation*}
f(b_\mu)\leq \lambda_1 f(x_1)+ \lambda_2 f(x_2)+ \lambda_3 f(x_3),
\end{equation*}
where $b_\mu$  is the baricenter of the measure  $\mu$.

In our case we have that $f(x_1)=f_1(a),\, f(x_2)=f_2(b), f(x_3)=f_3(c) $. In the general case, on a space $(N,d)$ the calculus of the baricenter of the measure  $\mu$ 
is given by the formula
\begin{equation*}
b_u= argmin_{z}\int_{N} d^{2}(z,x)-d^{2}(y,x)\; d\mu(x)=argmin_{z}\int_{N}
d^{2}(z,x)\; d\mu(x).
\end{equation*}

The baricenter $b_\mu$ can be situated on each edges  $N_1,
N_2,N_3$,  and depends on the minimum of the following three expresions
\begin{equation*}
argmin_{\{ r>=0\}} \Big(\lambda_1 (r+a)^2+ \lambda_2 (r+b)^2+
\lambda_3 (c-r)^2\Big),
\end{equation*}
\begin{equation*}
argmin_{\{ r>=0\}} \Big(\lambda_1 (r+a)^2+ \lambda_2 (b-r)^2+
\lambda_3 (r+c)^2\Big),
\end{equation*}
\begin{equation*}
argmin_{\{ r>=0\}} \Big( \lambda_1 (a-r)^2+ \lambda_2 (r+b)^2+
\lambda_3 (r+c)^2\Big).
\end{equation*}

An elementary calculus gives us the following cases:

a) If $\lambda_3 c\geq \lambda_1 a+ \lambda_2 b $ then
\begin{equation*}
b_\mu= (3,\lambda_3 c- \lambda_1 a- \lambda_2 b );
\end{equation*}

b) If  $\lambda_2 b\geq \lambda_1 a+ \lambda_3 c $ then
\begin{equation*}
b_\mu= (2,\lambda_2 b- \lambda_1 a- \lambda_3 c );
\end{equation*}

c) If $\lambda_1 a\geq \lambda_2 b+ \lambda_3 c $ then
\begin{equation*}
b_\mu= (1,\lambda_1 a- \lambda_2 b- \lambda_3 c );
\end{equation*}

d) Otherwise, we have that $b_\mu=o$.

We consider the last two cases,  the other two cases being similarly. 

 Let us suppose that $\lambda_1 a\geq \lambda_2 b+ \lambda_3 c $, where the baricenter is given by
\begin{equation*}
b_\mu= (1,\lambda_1 a- \lambda_2 b- \lambda_3 c ).
\end{equation*}

We will prove that
\begin{equation*}
f(b_\mu)= f_1(\lambda_1 a- \lambda_2 b- \lambda_3 c )\leq \lambda_1 f_1(a)+
\lambda_2 f_2(b)+\lambda_3 f_3(c),
\end{equation*}
which can be rewritten in the following form
\begin{equation}  \label{cneq6}
\lambda_1 a\Big( \frac{f_1(\lambda_1 a- \lambda_2 b- \lambda_3 c)}{\lambda_1
a- \lambda_2 b- \lambda_3 c}-\frac{f_1(a)}{a} \Big)\leq \lambda_2b \Big(
\frac{f_1(\lambda_1 a- \lambda_2 b- \lambda_3 c)}{\lambda_1 a- \lambda_2 b-
\lambda_3 c} + \frac{f_2(b)}{b} \Big)
\end{equation}
\begin{equation*}
+ \lambda_3 c \Big( \frac{f_1(\lambda_1 a- \lambda_2 b- \lambda_3 c)}{%
\lambda_1 a- \lambda_2 b- \lambda_3 c}+\frac{f_3(c)}{c} \Big).
\end{equation*}

By using \eqref{cneq3} and \eqref{cneq4} and the fact that each convex function $g$ verifies  that
\begin{equation*}
\frac{g(x)-g(0)}{x-0}\geq g_d^{^{\prime }}(0), \qquad ( x>0),
\end{equation*}
we obtain that
\begin{equation*}
\frac{f_1(\lambda_1 a- \lambda_2 b- \lambda_3 c)}{\lambda_1 a- \lambda_2 b-
\lambda_3 c}\geq f_{1d}^{^{\prime }}(0),
\end{equation*}
\begin{equation*}
\frac{f_2(b)}{b}\geq f_{2d}^{^{\prime }}(0).
\end{equation*}

Summing up the last two inequalities and by using \eqref{cneq4} we obtain the positivity of the first term from the right hand side.
Similarly, we can prove the positivity of the second term from the right hand side. Since the left hand side term is negative, we obtain easily that \eqref{cneq6} holds.

 If the baricenter $b_\mu =o$ then   the following conditions hold simultaneously
\begin{equation*}
\lambda_3 c\leq \lambda_1 a+ \lambda_2 b,
\end{equation*}
\begin{equation*}
\lambda_2 b\leq \lambda_1 a+ \lambda_3 c,
\end{equation*}
\begin{equation*}
\lambda_1 a\leq \lambda_2 b+ \lambda_3 c.
\end{equation*}

In this case we need to prove the following inequality
\begin{equation}
0 \leq \lambda_1 f_1(a)+ \lambda_2 f_2(b)+ \lambda_3 f_3(c).
\end{equation}

Since $f_1(a)\geq a f'_{1d}(0),\, f_2(b)\geq b f'_{2d}(0),\, f_3( c )\geq c f'_{3d}(0)$,  it is sufficiently to prove that 
\begin{equation*}
\lambda_1 a f'_{1d}(0)+ \lambda_2 b f'_{2d}(0)+
\lambda_3cf'_{3d}(0)\geq 0.
\end{equation*}

By using relations \eqref{cneq3}-\eqref{cneq5} we observe that at least two derivatives are positive and at most one derivate is negative. 
Without loosing the generality let us suppose that  $f'_{3d}(0)> f'_{2d}(0)>0,$  $f'_{1d}(0)<0$ and 
$\lambda_1 a\geq \lambda_2 b\geq \lambda_3 c$.  We have the following estimates

\begin{equation*}
\lambda_1a f'_{1d}(0)+ \lambda_2 b f'_{2d}(0)+
\lambda_3cf'_{3d}(0)\geq \lambda_1a f'_{1d}(0)+
f'_{2d}(0)(\lambda_2 b+ \lambda_3c)
\end{equation*}
\begin{equation*}
\geq \lambda_1 a (-f'_{2d}(0))+ f'_{2d}(0)(\lambda_2
b+ \lambda_3 c)= f'_{2d}(0)(\lambda_2b+\lambda_3 c- \lambda_1
a)\geq 0,
\end{equation*}
and the proof ends.
\end{proof}

\noindent\textbf{Acknowledgement}. The author  was supported by
the strategic grant POSDRU/159/1.5/ S/133255, Project ID 133255 (2014),
co-financed by the European Social Fund within the Sectorial Operational
Program Human Resources Development 2007 - 2013.

\end{document}